\documentclass[12pt]{article}
\usepackage{amsfonts}
\usepackage{latexsym}
\usepackage{cite}
\usepackage{amsmath,amsfonts,latexsym,amssymb}
\usepackage[mathscr]{eucal}
\usepackage{cases}
\usepackage{amsthm}

\usepackage[bf,small]{caption2}
\usepackage{float}
\usepackage{graphicx}
\usepackage{amsmath}
\usepackage{amssymb}
\usepackage[all]{xy}

\newtheorem{theorem}{theorem}[section]
\newtheorem{thm}[theorem]{Theorem}
\newtheorem{lem}[theorem]{Lemma}
\newtheorem{prob}[theorem]{Problem}
\newtheorem{prop}[theorem]{Proposition}

\newtheorem{alg}[theorem]{Algorithm}

\newtheorem{defn}[theorem]{Definition}

\newtheorem{rmk}[theorem]{Remark}

\begin{document}

\title{\textbf{How to find $G$-admissible abelian regular coverings of a graph?}}
\author{\Large Haimiao Chen
\footnote{Email: \emph{chenhm@math.pku.edu.cn}}\\
\normalsize \em{Department of mathematics, Peking University, Beijing, China}\\
\Large Hao Shen
\footnote{Email: \emph{shenhao208@mails.gucas.ac.cn} \ partially supported by NSFC project 61173009}\\
\normalsize \em{Institute of mathematics, Chinese Academy of Sciences, Beijing, China}}
\date{}
\maketitle

\begin{abstract}
Given a finite connected simple graph $\Gamma$, and a subgroup $G$ of its automorphism group, a general method for finding all finite
abelian regular coverings of $\Gamma$ that admit a lift of each element of $G$ is developed.
As an application, all connected arc-transitive abelian regular coverings of the Petersen graph are classified up to isomorphism.
\end{abstract}

\noindent\emph{Keywords:} abelian, regular covering of a graph, $G$-admissible, automorphism, lift, Petersen graph

\emph{MSC2010:} 05E18.

\section{Introduction}

In this paper we develop an effective method to find all finite abelian regular coverings (up to isomorphism) of a graph, admitting lifts of a given group of automorphisms of the graph. We then apply the method to classify arc-transitive finite abelian regular coverings of the Petersen graph. This extends the work of \cite{Crite} which classified arc-transitive elementary abelian regular coverings of the Petersen graph.

Recently, people have been paying much attention to symmetries of graphs. An initial motivation to find regular coverings admitting lifts of a given group of automorphisms of the base graph is to construct infinite families of graphs with specific symmetric properties. For example, new families of 2-arc-transitive graphs were constructed in \cite{2-arc} as long as 2-arc-transitive $\mathbb{Z}_{p}^{3}$-coverings were found. A systematic method was developed in \cite{Elementary,Vertex} to find isomorphism classes of elementary abelian coverings of a graph that admit lifts of given automorphisms; the method was applied in \cite{Vertex} to classify vertex-transitive elementary abelian coverings of the Petersen graph. Besides the ones mentioned above, there are many such results, see \cite{Q3, K5, Pappus, Heawood}, and so on.

But classification of general finite abelian regular coverings of a graph is never seen.
Theoretically, as pointed out in \cite{2-arc}, classification of general finite abelian regular coverings can be reduced to that of elementary abelian ones. But the difficulty is that, when elementary abelian regular coverings are classified, one still needs to go on to classify elementary abelian regular coverings of each of these covering graphs, and go on again. One is not able to stop at a step and say that the classification of abelian regular coverings of the initial graph is completed!

Our paper makes a breakthrough. The method is based on the criterion for lifting automorphisms proposed by the first author in \cite{chm}. We reduce the problem (Problem \ref{prob:initial})
of finding abelian coverings to that (Problem \ref{prob:final}) of finding numbers and matrices satisfying certain conditions.

\vspace{4mm}

The content is organized as follows.

Section 2 is a preliminary on topological graph theory. Section 3 is a framework suitable for general graphs and automorphisms; to determine all abelian regular coverings to which some given automorphisms can be lifted, a practical method is developed and written down as Algorithm \ref{alg}.
Section 4 is devoted to classifying all arc-transitive finite abelian regular coverings of the Petersen graph, and the result is given in Theorem \ref{thm:abelian coverings of Petersen}.

\vspace{4mm}

Finally, some notational conventions.

For a commutative ring $R$, let $R^{\times}$ be the set of units of $R$. Let $R^{n,m}$ denote the set of all $n\times m$ matrices with entries in $R$. Identify $R^{1,m}$ with $R^{m}$. For a matrix $U\in R^{n,m}$, let $U_{i,j}$ denote its $(i,j)$-th entry and write $U=(U_{i,j})_{n\times m}$; the row space of the matrix $U$ is denoted as $\langle U\rangle$. Let $\textrm{GL}(m,R)$ be the set of invertible $m\times m$ matrices in $R$. For $U\in R^{n,m}, C\in\mathbb{Z}^{k,n}, K\in\mathbb{Z}^{m,l}$, it is meaningful to talk about $CU\in R^{k,m}$ and $UK\in R^{n,l}$.

We use $\Sigma_{m}$ to denote the permutation group on $m$ letters. Each $\tau\in\Sigma_{m}$ is identified with a permutation matrix which is also denoted as $\tau$ and given by $\tau_{i,j}=\delta_{i,\tau(j)}$.

For $n\in\mathbb{Z}$ with $n>1$, denote the quotient group of $\mathbb{Z}$ modulo $n$ as $\mathbb{Z}_{n}$. Actually it is a quotient ring of $\mathbb{Z}$, and is a field when $n$ is prime. Note that there is a standard bijection $f_{n}:\mathbb{Z}_{n}\rightarrow\{0,\cdots,n-1\}$. We shall often identify $\lambda\in\mathbb{Z}_{n}$ with $f_{n}(\lambda)\in\mathbb{Z}$. For $\lambda,\lambda'\in\mathbb{Z}_{n}$, we say $\lambda>\lambda'$ if $f_{n}(\lambda)>f_{n}(\lambda')$.

\section{Preliminary}

In this section we recall some terminologies on graph theory and give some necessary definitions. For more one can refer to \cite{Crite, Voltage, Graph, Enum, Lift}.

In this paper all graphs are finite, connected, and simple (there is no loop or parallel edges). For a graph $\Gamma$, let $V(\Gamma), E(\Gamma), \textrm{Ar}(\Gamma)$ be the set of vertices, edges, arcs of $\Gamma$ respectively. For any $v\in V(\Gamma)$, we denote its neighborhood by $N(v)$.
Let $\{u,v\}\in E(\Gamma)$ denote the edge connecting the
vertices $u,v$, and let $(u,v)\in\textrm{Ar}(\Gamma)$ denote the arc from $u$ to $v$. Each edge $\{u,v\}$ gives rise to two
arcs $(u,v)$ and $(v,u)$ which are opposite to each other: $(v,u)=(u,v)^{-1}$.
A \emph{walk} $W=(v_{1},\cdots,v_{n})$ is a tuple such that $(v_{i},v_{i+1})\in\textrm{Ar}(\Gamma)$, for $1\leqslant i<n$.

A \emph{covering} $\pi:\Lambda\rightarrow\Gamma$ is a map of graphs such that $\pi:V(\Lambda)\rightarrow V(\Gamma)$ is surjective, and $\pi|_{N(\tilde{v})}\rightarrow \pi|_{N(v)}$ is bijective for any $v\in V(\Gamma)$ and $\tilde{v}\in \pi^{-1}(v)$. The group $A$ of automorphisms of $\Lambda$ which fix each fiber setwise is called the \emph{covering transformation group}. The covering is called \emph{regular} if $A$ acts on each fiber transitively.

For two coverings $\pi:\Lambda\rightarrow\Gamma, \pi':\Lambda'\rightarrow\Gamma$, an \emph{isomorphism} $\pi\rightarrow\pi'$
is a pair of isomorphisms $(\alpha,\tilde{\alpha})$ with $\alpha\in\textrm{Aut}(\Gamma)$, $\tilde{\alpha}:\Lambda\rightarrow\Lambda'$
such that $\alpha\circ\pi=\pi'\circ\tilde{\alpha}$. Call $\tilde{\alpha}$ a \emph{lift} of $\alpha$ along $(\pi,\pi')$, and say that $\alpha$ is lifted to $\Lambda\rightarrow\Lambda'$. When $\alpha=\textrm{id}$, call $(\alpha,\tilde{\alpha})$ an \emph{equivalence} and say that $\pi$ is equivalent to $\pi'$.

When $\Lambda=\Lambda'$ and $\pi=\pi'$, $\tilde{\alpha}$ is called a lift of $\alpha$ along $\pi$, and $\alpha$ is said to be lifted to $\Lambda$. If $G\leqslant\textrm{Aut}(\Gamma)$ such that each $\alpha\in G$ can be lifted to $\Lambda$, then $\pi$ is said to be \emph{$G$-admissible}.

Given a finite abelian group $A$, a \emph{voltage assignment} of $\Gamma$ in $A$ is a map $\phi:\textrm{Ar}(\Gamma)\rightarrow A$ such that
$\phi(x^{-1})=-\phi(x)$ for all $x\in\textrm{Ar}(\Gamma)$.
Each voltage assignment $\phi$ determines a regular covering $\Gamma\times_{\phi}A$ of $\Gamma$, having vertex set $V(\Gamma\times_{\phi}A)=V(\Gamma)\times A$ and edge set $E(\Gamma\times_{\phi}A)=\{\{(u,g),(v,\phi(u,v)+g)\}\colon \{u,v\}\in E(\Gamma)\}$; there is a canonical covering map given by the projection onto the first coordinate.

Let $T$ be an arbitrarily chosen spanning tree of $\Gamma$. A voltage assignment $\phi$ is called \emph{$T$-reduced} if $\phi(x)=0$ for all $x\in\textrm{Ar}(T)\subset\textrm{Ar}(\Gamma)$.
It is known (see \cite{Voltage}) that every regular covering of $\Gamma$ with covering transformation group $A$ is isomorphic to $\Gamma\times_{\phi}A$ for some $T$-reduced voltage assignment $\phi$.

For any walk $W=(v_{1},\cdots,v_{n})$, the voltage of $W$ is by definition $\phi(W)=\phi(v_{1},v_{2})+\cdots+\phi(v_{n-1},v_{n})$.

Choose a base vertex $v_{0}$. For each $v\in V(\Gamma)$ there is a unique reduced walk $W(v)$ in $T$
from $v_{0}$ to $v$. For each arc $(u,v)$ let $L(u,v)=W(u)\cdot(u,v)\cdot W(v)^{-1}$.

Let $e_{1},\cdots,e_{b}$ ($b$ is the first Betti number of the graph) be all the cotree edges, and for each $i$, choose an arc $x_{i}$ from $e_{i}$. Then the first homology group $H_{1}(\Gamma;\mathbb{Z})$ is the free abelian group generated by $L(x_{1}),\cdots,L(x_{b})$. There is a well-defined homomorphism
\begin{align}
E(\phi):H_{1}(\Gamma;\mathbb{Z})\rightarrow A,  \label{eq:homo}
\end{align}
which is surjective because $\Gamma\times_{\phi}A$ is assumed to be connected.

Each automorphism $\alpha\in\textrm{Aut}(\Gamma)$ induces an automorphism
\begin{align}
\alpha_{\ast}:H_{1}(\Gamma;\mathbb{Z})\rightarrow &H_{1}(\Gamma;\mathbb{Z}). \label{iso}
\end{align}

The following proposition comes from Theorem 3 in \cite{Enum} and its proof:
\begin{prop} \label{prop:lifting iso}
Given two $A$-coverings $\Gamma\times_{\phi}A$ and $\Gamma\times_{\psi}A$. An automorphism $\beta\in\rm Aut(\Gamma)$ can be lifted to $\Gamma\times_{\phi}A\rightarrow\Gamma\times_{\psi}A$ if and only if there exists a group automorphism $\sigma:A\rightarrow A$ such that
$E(\psi)\circ\beta_{\ast}=\sigma\circ E(\phi)$, or equivalently, $\beta_{\ast}(\ker(E(\phi)))=\ker(E(\psi))$.

In particular, $\Gamma\times_{\phi}A$ is equivalent to $\Gamma\times_{\psi}A$ if and only if $\ker(E(\phi))=\ker(E(\psi))$.
\end{prop}

As a corollary, we obtain a theoretical criteria for lifting automorphisms. See \cite{Crite} Proposition 1.2.

\begin{prop}\label{prop:condition-for-lift}
An automorphism $\alpha\in\rm Aut(\Gamma)$ can be lifted to $\Gamma\times_{\phi}A$ if and only if there is an automorphism $\sigma:A\rightarrow A$ such that $E(\phi)\circ\alpha_{\ast}=\sigma\circ E(\phi)$, or equivalently,
$\alpha_{\ast}(\ker(E(\phi)))=\ker(E(\phi))$.
\end{prop}

\section{General framework}

\subsection{Main goal}

We want to solve the following
\begin{prob} \label{prob:initial}
\rm Given a graph $\Gamma$ and a subgroup $G$ of $\rm Aut(\Gamma)$, find all $G$-admissible finite abelian regular coverings of $\Gamma$, up to isomorphism.
\end{prob}
\begin{defn} \label{defn}
\rm Given voltage assignments $\phi_{i}:\textrm{Ar}(\Gamma)\rightarrow A_{i}, 1\leqslant i\leqslant m$, there is a product voltage assignment $\phi=\prod\limits_{i=1}^{m}\phi_{i}:\textrm{Ar}(\Gamma)\rightarrow \prod\limits_{i=1}^{m}A_{i}$, given by
$$\phi(x)=(\phi_{1}(x),\cdots,\phi_{m}(x)).$$
Call $\Gamma\times_{\phi}(\prod\limits_{i=1}^{m}A_{i})$ the \emph{fibered product} of the coverings $\Gamma\times_{\phi_{i}}A_{i}$.
\end{defn}

Suppose $A=\prod\limits_{p\in I}A^{(p)}$ where $I$ is a finite set of prime numbers, and for each $p\in I$, $A^{(p)}=\prod\limits_{\eta=1}^{n_{p}}\mathbb{Z}_{p^{k(p,\eta)}}$ with $k(p,1)\geqslant\cdots\geqslant k(p,n_{p})\geqslant 1$.

Suppose $T$ is a spanning tree of $\Gamma$ as in Section 2, and $\phi:\textrm{Ar}(\Gamma)\rightarrow A$ is a voltage assignment which is assumed to be $T$-reduced as we always do. For each $p\in I$, let $q^{(p)}:A\rightarrow A^{(p)}$ denote the projection, and let $\phi^{(p)}=q^{(p)}\circ\phi$; call $\Gamma\times_{\phi^{(p)}}A^{(p)}$ the \emph{$p$-primary part} of $\Gamma\times_{\phi}A$.

It is easy to see that $E(\phi)$ is surjective if and only if $E(\phi^{(p)})$ is for all $p$, and $\alpha_{\ast}(\ker(E(\phi)))=\ker(E(\phi))$ if and only if $\alpha_{\ast}(\ker(E(\phi^{(p)})))=\ker(E(\phi^{(p)}))$ for all $p$. Hence we have
\begin{lem} \label{lem:reduction to primary}
The regular covering $\Gamma\times_{\phi}A$ is the fibered product of its primary parts; it admits a lift of $\alpha$ if and only if each primary part does.
\end{lem}
  Without loss of generality, from now on we assume that $A$ is an abelian $p$-group for some prime $p$: $A=\prod\limits_{\eta=1}^{n}\mathbb{Z}_{p^{k(\eta)}}$ with
$k(1)\geqslant\cdots\geqslant k(n)\geqslant 1$.

For any voltage assignment $\phi:\textrm{Ar}(\Gamma)\rightarrow A$, it is clear that $p^{k(1)}\mathbb{Z}^{b}\leqslant\ker E(\phi)$. Let
\begin{align}
\overline{\ker}(E(\phi))=\ker E(\phi)/p^{k(1)}\mathbb{Z}^{b}\leqslant\mathbb{Z}_{p^{k(1)}}^{b}.
\end{align}
\begin{defn}
\rm A \emph{$G$-admissible pair} $(k,K)$ consists of a positive integer $k$ and a subgroup $K\leqslant\mathbb{Z}_{p^{k}}^{b}$ such that $\alpha_{\ast}(K)=K$ for all $\alpha\in G$, where $\alpha_{\ast}$ is the induced action of $\alpha\in\rm Aut(\Gamma)$ on $H_{1}(\Gamma;\mathbb{Z}_{p^{k}})\cong\mathbb{Z}_{p^{k}}^{b}$.

Say two pairs $(k,K)$ and $(k',K')$ are \emph{isomorphic} if $k=k'$ and there exists $\beta\in\textrm{Aut}(\Gamma)$ such that $\beta_{\ast}(K)=K'$.
\end{defn}
\begin{rmk} \label{rmk:pair}
\rm From Proposition \ref{prop:lifting iso} and Proposition \ref{prop:condition-for-lift}, we see that Problem \ref{prob:initial} can be reduced to finding $G$-admissible pairs up to isomorphism.

Be careful that when $(k,K)$ is $G$-admissible, it is not guaranteed that $(k,\beta_{\ast}(K))$ also is, unless $G$ is normal in $\textrm{Aut}(\Gamma)$.
\end{rmk}

\subsection{An algebraic preparation}

Let $p$ be a prime number, $k\geqslant 1$, and let $R=\mathbb{Z}_{p^{k}}$.
\begin{defn}
\rm For $\lambda\in R$, the \emph{$p$-degree} of $\lambda$, denoted $\deg_{p}(\lambda)$, is the unique integer $r$, $0\leqslant r\leqslant k$
such that $\lambda=p^{r}\cdot\chi$, with $\chi\in R^{\times}$.
\end{defn}

The following theorem, proved in \cite{chm}, plays a fundamental role. We include the proof here for the reader's convenience.
\begin{thm} \label{thm:foundation}
For each $X\in R^{n,m}$, there exist $Q\in\rm GL(\it n,R)$, $S\in\rm GL(\it m,R)$ such that $X^{0}:=QXS$ is ``in normal form",
that is, $(X^{0})_{i,j}=\delta_{i,j}\cdot p^{r_{i}}$, with $0\leqslant r_{1}\leqslant\cdots\leqslant r_{n}\leqslant k$.
\end{thm}
\begin{proof}
If $X=0$, there is nothing to show. So let us assume $X\neq 0$.

Choose an entry $X_{i_{1},j_{1}}\neq 0$ with smallest $p$-degree, and suppose $X_{i_{1},j_{1}}=p^{r_{1}}\cdot\chi$ with $\chi$ invertible. Interchange the $i_{1}$-th row of $X$ with the first row, and the $j_{1}$-th column with the first column, and divide the first row by $\chi$. The matrix obtained has $(1,1)$-entry $p^{r_{1}}$, dividing all entries. Then perform row-transformations to eliminate the $(i,1)$-entries for $1<i\leqslant n$. Denote the new matrix by $X^{(1)}$.

Do the same thing to the $(n-1)\times(m-1)$ down-right minor of $X^{(1)}$, and then go on. At each step, we can take row transformations to eliminate elements in a column below the main diagonal, and rearrange the diagonal elements. At $l$-th step for some $l\leqslant\min\{m,n\}$, we get some matrix $X^{(l)}$ whose $i$-th diagonal entry is $p^{r_{i}}$ for $1\leqslant i\leqslant l$, with $r_{1}\leqslant\cdots\leqslant r_{l}$, and the elements under the main diagonal all vanish.

Finally, perform column transformations to $X^{(l)}$, to eliminate all the ``off-diagonal" entries. Then the resulting matrix is in normal form.
\end{proof}

\begin{rmk} \label{rmk:normal form}
\rm From the proof it is easy to see that $S$ can be chosen to be $\omega S'$ for some upper-triangular matrix $S'$ and some
permutation matrix $\omega$.
\end{rmk}
\begin{thm} \label{thm:subgroup}
Each subgroup $C\leqslant R^{m}$ is of the form $\langle PQ\omega\rangle$ such that (a) $P\in R^{l,m}$,
$l\leqslant m$, $P_{i,j}=\delta_{i,j}\cdot p^{r_{i}}, 0\leqslant r_{1}\leqslant\cdots\leqslant r_{l}<k$; (b) $Q\in\rm GL(\it m,R)$
is upper-triangular with diagonal 1's, and $0\leqslant Q_{i,j}<p^{r_{j}-r_{i}}$ for all $j>i$, (by convention set $r_{i}=k$ for $i>l$); (c) $\omega\in\Sigma_{m}$. Moreover, $C\cong\mathbb{Z}_{p^{k-r_{l}}}\times\cdots\times\mathbb{Z}_{p^{k-r_{1}}}$, hence $P$ is uniquely determined by $C$.
\end{thm}
\begin{proof}
Choose a set of generators $\{X_{1},\cdots,X_{l'}\}$ of $C$ and make the matrix $X\in R^{l',m}$, with row vectors $X_{1},\cdots, X_{l'}$. By Theorem \ref{thm:foundation} and Remark \ref{rmk:normal form}, there exist $Y\in\textrm{GL}(l',R), S'\in\textrm{GL}(m,R),\omega'\in\Sigma_{m}$ such that $S'$ is upper-triangular, and $P':=YX\omega S'$ takes the form $P'_{i,j}=\delta_{i,j}\cdot p^{r_{i}}$, with $0\leqslant r_{1}\leqslant\cdots\leqslant r_{l'}\leqslant k$. Let $l$ be the maximal $i$ with $r_{i}<k$, and let $P$ be the matrix obtained by taking the first $l$ rows of $P'$. Note that via scalar multiplication and division with remainder, we are able to take elementary row transformations to convert $P{S'}^{-1}$ into $PQ$ for some $Q\in\textrm{GL}(m,R)$ satisfying (b). Then $C=\langle X\rangle=\langle YX\rangle=\langle PQ\omega\rangle$ with $\omega=\omega'^{-1}$.

Now verify the last assertion. For a general matrix $B$, let $B_{i}$ denote its $i$-th row vector. Define a homomorphism $f:C=\langle PQ\omega\rangle\rightarrow\langle P\rangle\cong\mathbb{Z}_{p^{k-r_{l}}}\times\cdots\times\mathbb{Z}_{p^{k-r_{1}}}$ by sending $(PQ\omega)_{i}$ to $P_{i}$, $1\leqslant i\leqslant l$. It is well-defined, because for all $w=(w_{1},\cdots,w_{l})\in\mathbb{Z}^{l}$, we have $wPQ\omega=0$ if and only if $wP=0$. Clearly $f$ is an isomorphism.
\end{proof}
\begin{defn} \label{defn:P-equivalence}
\rm Suppose $l\leqslant m$, $P\in R^{l,m}$ with $P_{i,j}=\delta_{i,j}\cdot p^{r_{i}}$ with $0\leqslant r_{1}\leqslant\cdots\leqslant r_{l}<k$.
Two matrices $M,M'\in\textrm{GL}(m,R)$ are called \emph{$P$-equivalent} and denoted as $M\sim_{P}M'$, if $\langle PM\rangle=\langle PM'\rangle$.
\end{defn}

\begin{lem} \label{lem:equivalent}
Let $P$ be given in Definition \ref{defn:P-equivalence}.
Then  $M\sim_{P}M'$ is equivalent to $\deg_{p}(MM'^{-1})_{i,j}\geqslant r_{j}-r_{i}$ for all $j>i$, (set $r_{l+1}=\cdots =r_{m}=k$ by convention).

\end{lem}
\begin{proof}
Define a homomorphism $\langle P\rangle\rightarrow\langle P(MM'^{-1})\rangle$ by $w\mapsto w(MM'^{-1})$. This is an isomorphism, hence
the finite group $\langle P(MM'^{-1})\rangle$ contains as many elements as $\langle P\rangle$ does. So $M\sim_{P}M'$, which is equivalent to $\langle P(MM'^{-1})\rangle=\langle P\rangle$, holds if and only if
$\langle P(MM'^{-1})\rangle\leqslant\langle P\rangle$. The later condition is equivalent to $\deg_{p}(MM'^{-1})_{i,j}\geqslant r_{j}-r_{i}$ for all $j>i$,
using the characterization that $(w_{1},\cdots,w_{m})\in\langle P\rangle$ if and only if $\deg_{p}(w_{j})\geqslant r_{j}$ for all $j$.
\end{proof}

\subsection{The method of finding admissible abelian coverings}

Recall Remark \ref{rmk:pair} that we are going to find $G$-admissible pairs $(k,K)$ up to isomorphism. For each $k$, there is always an admissible pair $(k,\mathbf{0})$ where $\mathbf{0}$ is the trivial subgroup. In the rest of this section, let us assume $K\neq\mathbf{0}$.

For $\alpha\in\textrm{Aut}(\Gamma)$, let $S^{\alpha}\in\textrm{GL}(b,\mathbb{Z})$ denote the matrix corresponding to the induced isomorphism (\ref{iso}) of $\alpha$, so that
\begin{align}
\alpha_{\ast}(L(x_{i}))=\sum\limits_{j=1}^{b}(S^{\alpha})_{i,j}L(x_{j}), \hspace{5mm} i=1,\cdots,b. \label{induced matrix}
\end{align}
Abusing the notation, we denote the image of $S^{\alpha}$ under the homomorphism
$\textrm{GL}(b,\mathbb{Z})\rightarrow\textrm{GL}(b,\mathbb{Z}_{p^{k}})$
induced by the quotient map $\mathbb{Z}\rightarrow\mathbb{Z}_{p^{k}}$ also by $S^{\alpha}$. Now if $K=\langle PQ\omega\rangle\leqslant\mathbb{Z}_{p^{k}}^{b}$, then $\alpha_{\ast}(K)=\langle PQ\omega S^{\alpha}\rangle$.

\begin{defn}
\rm A \emph{$G$-admissible solution} is a tuple $(p;k,l;r_{1},\cdots,r_{l};Q,\omega)$ consisting of integers $k\geqslant 1$, $0<l<b, 0\leqslant r_{1}\leqslant\cdots\leqslant r_{l}<k$, a prime numbers $p$, matrices $Q\in\textrm{GL}(b,\mathbb{Z}_{p^{k}}), \omega\in\Sigma_{b}$, satisfying the following condition (by convention set $r_{i}=k$ for $l<i\leqslant b$):
$$
\leqno(\star)\hspace{3mm}
\begin{split}
&Q \text{ is upper-triangular with diagonal }1's, 0\leqslant Q_{i,j}<p^{r_{j}-r_{i}} \text{for all }i<j; \\
&\text{ for all } \alpha\in G, Q\omega S^{\alpha}\sim_{P}Q\omega, \text{ where } P\in\mathbb{Z}_{p^{k}}^{l,b}, P_{i,j}=\delta_{i,j}\cdot p^{r_{i}}.
\end{split}
$$
Call two solutions $(p;k,l;r_{1},\cdots,r_{l};Q,\omega)$ and $(p';k',l';r'_{1},\cdots,r'_{l};Q',\omega')$ \emph{isomorphic} if $p=p', k=k', l=l', r_{i}=r'_{i}, 1\leqslant i\leqslant l$, and $Q'\omega'\sim_{P}Q\omega S^{\beta}$ for some $\beta\in\textrm{Aut}(\Gamma)$.
\end{defn}

By Theorem \ref{thm:subgroup} and Lemma \ref{lem:equivalent}, $G$-admissible pairs are the same as $G$-admissible solutions. Moreover, two solutions are isomorphic if and only if the corresponding pairs and the corresponding regular coverings are isomorphic. Thus Problem \ref{prob:initial} is further reduced to
\begin{prob} \label{prob:final}
\rm Find all $G$-admissible solutions up to isomorphism.
\end{prob}
\begin{rmk}   \label{rmk:from solution to covering}
\rm Each solution determines, up to equivalence, a unique connected $G$-admissible covering $\Gamma\times_{\phi}(\prod\limits_{\eta=1}^{n}\mathbb{Z}_{p^{k(\eta)}})$ with $n=b-i_{0}$ ($i_{0}$ being the largest $i$ with $r_{i}=0$), $k(\eta)=r_{b+1-\eta}, 1\leqslant\eta\leqslant n$, and
\begin{align}
\phi(x_{i})=(((Q\omega)^{-1})_{i,b},\cdots,((Q\omega)^{-1})_{i,i_{0}+1})\in\mathbb{Z}_{p^{k(1)}}\times\cdots\times\mathbb{Z}_{p^{k(n)}}, \hspace{5mm} 1\leqslant i\leqslant b;
\end{align}
this is because there are isomorphisms
$$\mathbb{Z}^{b}_{p^{k(1)}}/\langle PQ\omega\rangle\cong\mathbb{Z}^{b}_{p^{k(1)}}/\langle P\rangle, \hspace{5mm} \overline{w}\mapsto\overline{w(Q\omega)^{-1}},$$
and
\begin{align*}
\mathbb{Z}^{b}_{p^{k(1)}}/\langle P\rangle&\cong\prod\limits_{\eta=1}^{n}\mathbb{Z}_{p^{k(\eta)}}, \\
\overline{(u_{1},\cdots,u_{b})}&\mapsto(u_{b}\pmod{p^{k(1)}},\cdots,u_{i_{0}+1}\pmod{p^{k(n)}}),
\end{align*}
where by $u\pmod{p^{r}}$ for $u\in\mathbb{Z}_{p^{k}}, r\leqslant k$, we mean the image of $u$ under the canonical homomorphism
\begin{align} \label{eq:canonical homomorphism}
\mathbb{Z}_{p^{k}}\rightarrow\mathbb{Z}_{p^{r}}.
\end{align}
\end{rmk}
\vspace{4mm}

The second part of ($\star$) is, by Lemma \ref{lem:equivalent}, equivalent to
\begin{align}
\deg_{p}(Q\omega S^{\alpha}(Q\omega)^{-1})_{i,j}\geqslant r_{j}-r_{i} \hspace{5mm} \text{for all }j>i \text{ and }\alpha\in G.
\end{align}

It may be frustrating that there are so many variables, especially that $\omega$ may take $b!$ values. Nevertheless, we have some key observations which help simplify the calculations very much.

\begin{prop}\label{prop:key1}
Suppose that $(p;k,l;r_{1},\cdots,r_{l};Q,\omega)$ is a solution.

(a) If $\sigma\in\Sigma_{b}$ satisfies $r_{\sigma(i)}=r_{i}$ for all $i\in\{1,\cdots,b\}$, then $(p;k,l;r_{1},\cdots,r_{l};\\ \sigma Q\sigma^{-1},\sigma\omega)$ is a solution isomorphic to $(p;k,l;r_{1},\cdots,r_{l};Q,\omega)$.

(b) If there exists $\beta\in\rm Aut(\Gamma)$ such that $\beta G\beta^{-1}=G$, and $S^{\beta}=D\tau$ with $\tau\in\Sigma_{b}$ and $D$ diagonal, then $(p;k,l;r_{1},\cdots,r_{l};Q',\omega\tau)$ is a solution isomorphic to $(p;k,l;r_{1},\cdots,r_{l};Q,\omega)$, for some $Q'$.
\end{prop}
\begin{proof}
(a) It is easy to see that $(\sigma Q\sigma^{-1},\sigma\omega)$ satisfies ($\star$). Since $r_{\sigma(i)}=r_{i}$, we have $\sigma\sim_{P}\textrm{id}$, hence $(\sigma Q\sigma^{-1})(\sigma\omega)=\sigma(Q\omega)\sim_{P}Q\omega$.

(b) Suppose $D_{i,j}=\delta_{i,j}\cdot d_{j}$. Let $D'$ be the diagonal matrix with $D'_{i,j}=\delta_{i,j}\cdot d_{\omega^{-1}(j)}$. Similarly to the proof of Theorem \ref{thm:subgroup}, we may take row transformations to convert $PQD'$ into $PQ'$ with $Q'\in\textrm{GL}(b,\mathbb{Z}_{p^{k}})$ satisfying the first part of ($\star$), then $Q'\sim_{P}QD'$. Hence $Q\omega S^{\beta}=Q\omega D\tau=QD'\omega\tau\sim_{P}Q'\omega\tau$. This also shows that $(p;k,j_{0};s_{1},\cdots,s_{j_{0}};Q',\omega\tau)$ is a solution: for all $\alpha\in G$, we have $Q\omega S^{\beta}S^{\alpha}=Q\omega S^{\beta\alpha\beta^{-1}}S^{\beta}\sim_{P}Q\omega S^{\beta}$, hence $Q'\omega\tau S^{\alpha}\sim_{P}Q'\omega\tau$.
\end{proof}

Let
\begin{align}
\textrm{Aut}_{G}(\Gamma;T)=\{\beta\in\textrm{Aut}(\Gamma)\colon\beta(T)=T,\beta G\beta^{-1}=G\}.
\end{align}

Note that for each $\beta\in\textrm{Aut}_{G}(\Gamma;T)$,
\begin{align}
S^{\beta}=D^{\beta}\tau^{\beta}, \label{eq:tau}
\end{align}
where $D^{\beta}$ is some diagonal matrix and $\tau^{\beta}\in\Sigma_{b}$ is the induced permutation on cotree edges.

Suppose $r_{i_{\eta-1}+1}=\cdots =r_{i_{\eta}}<r_{i_{\eta}+1}$, $0\leqslant\eta\leqslant s$, where $i_{s}=l$, and by convention we have set $i_{-1}=0, i_{s+1}=b$.
Define an equivalence relation $\sim$ on $\Sigma_{b}$ by declaring $\omega\sim\omega'$ if there exists
$\beta\in\textrm{Aut}_{G}(\Gamma;T)$ such that
\begin{align} \label{eq:equivalence}
\{(\omega\tau^{\beta})^{-1}(i)\colon i_{\eta-1}<i\leqslant i_{\eta}\}=\{\omega'^{-1}(i)\colon i_{\eta-1}<i\leqslant i_{\eta}\}, \hspace{5mm} \eta=0,\cdots,s+1.
\end{align}
Let $\Sigma_{b}(i_{0},\cdots,i_{s})$ be the set of equivalence classes.

By Proposition \ref{prop:key1}, to find isomorphism classes of $G$-admissible solutions with $r_{i_{\eta-1}+1}=\cdots =r_{i_{\eta}}<r_{i_{\eta}+1}$, it is sufficient to take a representative from each equivalence class in $\Sigma_{b}(i_{0},\cdots,i_{s})$.

\vspace{4mm}

For $r\leqslant k$ and $X\in\mathbb{Z}_{p^{k}}^{n,m}$, let $X\pmod{p^{r}}\in\mathbb{Z}_{p^{r}}^{n,m}$ denote the matrix obtained by replacing each entry by its image under (\ref{eq:canonical homomorphism}).

\begin{prop}\label{prop:key2}
Suppose that $(p;k,l;r_{1},\cdots,r_{l};Q,\omega)$ is a solution, $r_{i}<r_{i+1}$ for some $i$ with $i_{0}\leqslant i\leqslant l$. Then for any $r$ with $1\leqslant r\leqslant r_{i+1}-r_{i}$, there exists $\overline{Q}\in\rm GL(\it b,\mathbb{Z}_{p^{r}})$ such that $(p;r,i;0,\cdots,0;\overline{Q},\omega)$ is a solution.

Explicitly, write $Q$ in block form as
$Q=\left(
                             \begin{array}{cc}
                               Q_{1} & Q_{2} \\
                               \mathbf{0} & Q_{3}  \\
                             \end{array}
                           \right)
$,
where $\mathbf{0}\in\mathbb{Z}_{p^{k}}^{b-i,i}$ is the zero matrix, $Q_{1}\in$\rm GL$(i,\mathbb{Z}_{p^{k}})$, $Q_{2}\in\mathbb{Z}_{p^{k}}^{i,b-i}$, $Q_{3}\in$\rm GL$(b-i,\mathbb{Z}_{p^{k}})$, then
\begin{align}
\overline{Q}=\left(
\begin{array}{cc}
I & Q_{1}^{-1}Q_{2} \\
\mathbf{0} & I  \\
\end{array}
\right)\pmod{p^{r}}\hspace{5mm}\in\mathbb{Z}_{p^{r}}^{b,b}.
\end{align}
\end{prop}
\begin{proof}
The first half of ($\star$) is clearly satisfied. For the second half, writing
$Q\omega S^{\alpha}\omega^{-1}Q^{-1}\pmod{p^{r}}=
\left(
     \begin{array}{cc}
     \ast  & \mathbf{0} \\
     \ast & V^{\alpha} \\
     \end{array}
\right)$
with $V^{\alpha}\in\textrm{GL}(b-i,\mathbb{Z}_{p^{r}})$, and regarding $\omega, S^{\alpha}$ as in $\textrm{GL}(b,\mathbb{Z}_{p^{r}})$,
we have
$\overline{Q}\omega S^{\alpha}\omega^{-1}\overline{Q}^{-1}=
\left(
    \begin{array}{cc}
     \ast & \mathbf{0} \\
     \ast & Q_{3}^{-1}V^{\alpha}Q_{3} \\
    \end{array}
\right)$.
Thus $(\star)$ holds.
\end{proof}

\begin{rmk}
\rm The special case $r=1$ implies that, if $r_{i}<r_{i+1}$, then $(p;k,l;r_{1},\cdots,r_{l};Q,\omega)$ cannot be a $G$-admissible solution unless there exists a $G$-admissible $\mathbb{Z}_{p}^{i}$-covering of $\Gamma$. This is practically useful as long as the classification of $G$-admissible elementary abelian coverings is known.
\end{rmk}

Summarizing above discussions, now we go to solve Problem \ref{prob:final} by
\begin{alg} \label{alg}
\rm (a) Set $s=0, k=1$ and run Step (c) and (d) below to find elementary abelian regular coverings. That is to say, for each $i_{0}\in\{0,1,\cdots,b-1\}$ and each class in $\Sigma_{b}(i_{0})$, choose a representative $\omega$, then find all upper-triangular matrices $Q\in\textrm{GL}(b,\mathbb{Z}_{p})$ satisfying $Q_{i,i}=1, 1\leqslant i\leqslant b$, and $(Q\omega S^{\alpha}\omega^{-1}Q^{-1})_{i,j}=0$ for all $i\leqslant i_{0}<j$ and $\alpha\in G$.

Let $\mathcal{P}_{0}$ be the set of pairs $(p,i_{0})$ such that $(p;1,i_{0};0,\cdots,0;Q,\omega)$ emerges as a solution, i.e., there is a $G$-admissible $\mathbb{Z}_{p}^{b-i_{0}}$-covering of $\Gamma$.

(b) Determine the set
\begin{align*}
\mathcal{P}:=\{(p;i_{0},i_{1},\cdots,i_{s})\colon 0\leqslant i_{0}<i_{1}<\cdots< i_{s}<b,(p,i_{\eta})\in\mathcal{P}_{0},0\leqslant\eta\leqslant s\}.
\end{align*}

(c) For each $(p;i_{0},\cdots,i_{s})\in\mathcal{P}$, determine $\Sigma_{b}(i_{0},\cdots,i_{s})$.

For each equivalence class in $\Sigma_{b}(i_{0},\cdots,i_{s})$, choose a representative $\omega$, and find all $(c_{1},\cdots,c_{s},k)\in\mathbb{Z}^{s+1}$ and upper-triangular matrices $Q\in \textrm{GL}(b,\mathbb{Z}_{p^{k}})$ satisfying

(i) $1\leqslant c_{1}<\cdots <c_{s}<k$;

(ii) $Q_{i,i}=1, 1\leqslant i\leqslant b$; $0\leqslant Q_{i,j}<p^{c_{\eta}-c_{\eta'}}$ for all $i,j$ with $i_{\eta'}<i\leqslant i_{\eta'+1}\leqslant i_{\eta}<j\leqslant i_{\eta+1}$;

(iii) $(Q\omega S^{\alpha}\omega^{-1}Q^{-1})_{i,j}\equiv 0\pmod{p^{c_{\eta}-c_{\eta'}}}$, for all $i,j$ with $i_{\eta'}<i\leqslant i_{\eta'+1}\leqslant i_{\eta}<j\leqslant i_{\eta+1}$ and
all $\alpha\in G$.

Then $(p;k,l;r_{1},\cdots,r_{l};Q,\omega)$ is a solution, where $l=i_{s}$, $r_{i}=0$ for $i\leqslant i_{0}$, and $r_{i}=c_{\eta}$ for $i_{\eta-1}<i\leqslant i_{\eta}$.

(d) Distinguish the isomorphism classes among all the solutions.
\end{alg}
\begin{rmk}
\rm The goal of Step (a) can be achieved alternatively by applying the method of \cite{Elementary, Vertex}.
\end{rmk}
\begin{rmk} \label{rmk:congruence equation}
\rm In Step (c), when checking condition (c)(iii), it is sufficient to do it for a set of generators of $G$.

When $s=0$, (in which case we call the solution ``\emph{pure}"), by (c)(ii), $Q$ can be written in block form as
$\left(
    \begin{array}{cc}
        I & \Theta \\
        0 & I \\
    \end{array}
 \right)
$
for $\Theta\in\mathbb{Z}_{p^{k}}^{i_{0},b-i_{0}}$.
Writing $\omega S^{\alpha}\omega^{-1}=
\left(
    \begin{array}{cc}
        B_{1}^{\alpha} & B_{2}^{\alpha} \\
        B_{3}^{\alpha} & B_{4}^{\alpha} \\
    \end{array}
\right)
$,
condition (c)(iii) is equivalent to the equations over $\mathbb{Z}_{p^{k}}$:
\begin{align}
(B_{1}^{\alpha}+\Theta B_{3}^{\alpha})\Theta=B_{2}^{\alpha}+\Theta B_{4}^{\alpha}. \label{eq:congruence equation}
\end{align}

One obtains a system of quadratic congruence equations from (\ref{eq:congruence equation}) for various $\alpha$. It is convenient to first get pure solutions and then apply Proposition \ref{prop:key2} to obtain some constraints on $p,k,l,r_{1},\cdots,r_{l},Q,\omega$.
\end{rmk}
\begin{rmk} \label{rmk:reduction}
\rm If $(p;k,l;r_{1},\cdots,r_{l};Q,\omega)$ is a solution as in (c) such that $i_{0}=0$ (equivalently, $r_{1}>0$), then $(p;k',l;r'_{1},\cdots,r'_{l};Q,\omega)$ with $k'=k-c_{1}, r'_{i}=r_{i}-c_{1}$ is also a solution. Conversely, from any positive integer $c$ and any solution $(p;k,l;r_{1},\cdots,r_{l};Q,\omega)$, one can construct another solution $(p;k'',l;r''_{1},\cdots,r''_{l};Q,\omega)$ with $k''=k+c$, $r''_{i}=r_{i}+c$, $r''_{1}>0$.

Therefore we can consider the cases with $i_{0}>0$ at first and then generate those with $i_{0}=0$.
\end{rmk}

\section{Arc-transitive finite abelian coverings of the Petersen graph}

\subsection{The Petersen graph}

Let $X=\{a,b,c,d,e\}$, and $V=\{0,1,2,3,4,5,6,7,8,9\}$, where $0=\{a,b\}$, $1=\{c,d\}$, $2=\{c,e\}$, $3=\{d,e\}$, $4=\{a,e\}$, $5=\{b,e\}$, $6=\{a,d\}$, $7=\{b,d\}$, $8=\{a,c\}$, $9=\{b,c\}$. The Petersen graph has $V(\Gamma)=V$ and $E(\Gamma)=\{\{i,j\}|i,j\in V,i\cap j=\emptyset\}$. It is known that $\textrm{Aut}(\Gamma)\cong\Sigma_{5}$, the action being induced from that on $X$.

Fix a spanning tree $T$ as in Figure 1(a); the induced subgraph is shown in (b). For each cotree edge $\{i,j\}$, choose the arc $(i,j)$ with $i<j$. Label $x_{1}=(5,8),x_{2}=(7,8),x_{3}=(4,7),x_{4}=(4,9),x_{5}=(6,9),x_{6}=(5,6)$.

\begin{figure}[h]
  \centering
  \includegraphics[width=0.6\textwidth]{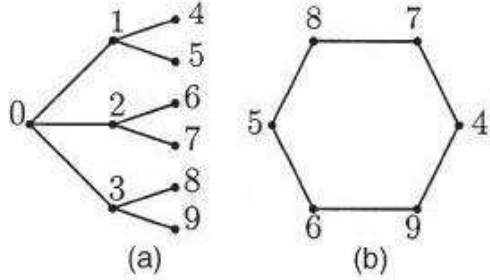}\\
  \caption{(a) The spanning tree T; (b) the induced subgraph}\label{1}
\end{figure}

Let $\alpha_{1}=(13)(67)(49)(58)$, $\alpha_{2}=(19)(56)(28)(03)$, $\alpha_{3}=(123)(468)(579)$, $\alpha_{4}=(45)(67)(89)$, induced by $(ab)(ce), (ae)(bd), (ced), (ab)$ respectively.

According to \cite{Crite}, $\langle\alpha_{1},\alpha_{2},\alpha_{3},\alpha_{4}\rangle=\textrm{Aut}(\Gamma)$, and the subgroup $H:=\langle\alpha_{1},\alpha_{2},\alpha_{3}\rangle$ is isomorphic to the alternating group on 5 elements, hence is normal. A regular covering $\tilde{\Gamma}$ of $\Gamma$ is arc-transitive if and only if $H$ can be lifted. Moreover, $\tilde{\Gamma}$ is 2-arc-transitive if $H$ can be lifted, and $\tilde{\Gamma}$ is 3-arc-transitive if the whole $\textrm{Aut}(\Gamma)$ can be lifted.

The matrices determined by (\ref{induced matrix}) can be calculated:\\
$S^{\alpha_{1}}=\left(\begin{array}{cccccc}
-1 & 0 & 0 & 0 & 0 & 0\\ 0 & 0 & 0 & 0 & 0 & -1 \\ 0 & 0 & 0 & 0 & -1 & 0 \\ 0 & 0 & 0 & -1 & 0 & 0 \\ 0 & 0 & -1 & 0 & 0 & 0 \\ 0 & -1 & 0 & 0 & 0 & 0 \end{array}\right)$,
$S^{\alpha_{2}}=\left(\begin{array}{cccccc}
0 & 0 & 0 & 0 & -1 & 0\\ 0 & -1 & 0 & 0 & 0 & 0 \\ 0 & 1 & 1 & -1 & 0 & 0\\ 0 & 0 & 0 & -1 & 0 & 0\\ -1 & 0 & 0 & 0 & 0 & 0 \\ 1 & 0 & 0 & 0 & -1 & -1
\end{array}\right)$,\\
$S^{\alpha_{3}}=\left(\begin{array}{cccccc}
0 & 0 & -1 & 0 & 0 & 0\\ 0 & 0 & 0 & -1 & 0 & 0 \\ 0 & 0 & 0 & 0 & 1 & 0 \\ 0 & 0 & 0 & 0 & 0 & -1 \\ -1 & 0 & 0 & 0 & 0 & 0 \\ 0 & 1 & 0 & 0 & 0 & 0 \end{array}\right)$,
$S^{\alpha_{4}}=\left(\begin{array}{cccccc}
0 & 0 & 0 & 1 & 0 & 0\\ 0 & 0 & 0 & 0 & 1 & 0 \\ 0 & 0 & 0 & 0 & 0 & 1 \\ 1 & 0 & 0 & 0 & 0 & 0 \\ 0 & 1 & 0 & 0 & 0 & 0 \\ 0 & 0 & 1 & 0 & 0 & 0 \end{array}\right)$.

\subsection{Arc-transitive abelian coverings}

Recall Theorem 3.2 of \cite{Crite}:
\begin{thm} \label{thm:recall}
Each connected arc-transitive elementary abelian covering of $\Gamma$ is isomorphic to a unique one given in Table 1. Moreover, $X(2,1), X(2,2),\\ X(5,3), X(p,6)$ are 3-arc-transitive, and $X'(2,1), X^{\pm}(p,3)$ are 2-transitive but not 3-arc-transitive.
\end{thm}
\begin{rmk}
\rm As pointed out in \cite{Crite}, $X^{+}(p,3)$ is not isomorphic to $X^{-}(p,3)$ as coverings, although they are isomorphic as graphs. Therefore, it is better to distinguish them.
\end{rmk}

We give new families of coverings in Table 2.
\begin{rmk}
\rm In Table 1, $\lambda_{\pm}$ are the two solutions of the equation $\lambda^{2}-\lambda-1=0$ in $\mathbb{Z}_{p}$. In Table 2, the $\lambda_{\pm}$ in 4th row are the two solutions of the equation $\lambda^{2}-\lambda-1=0$ in $\mathbb{Z}_{p^{k}}$, while the ones in 6th row are the two solutions of the equation $\lambda^{2}-\lambda-1=0$ in $\mathbb{Z}_{p^{s}}$ and are identified with elements of $\mathbb{Z}_{p^{k}}$ via the map $\mathbb{Z}_{p^{s}}\xrightarrow[]{f_{p^{s}}}\{0,1,\cdots,p^{s}-1\}\subset\{0,1,\cdots,p^{k}-1\}\xrightarrow[]{f^{-1}_{p^{k}}}\mathbb{Z}_{p^{k}}$, (for $f_{n}$ see the end of Section 1, and it should be warned that $(f_{p^{k}}^{-1}f_{p^{s}})(\lambda_{\pm})$ need not be solutions to the equation $\lambda^{2}-\lambda-1=0$ over $\mathbb{Z}_{p^{k}}$).
\end{rmk}

Here is the main result of this section:
\begin{thm} \label{thm:abelian coverings of Petersen}
Each connected primary arc-transitive finite abelian covering of $\Gamma$ is isomorphic to a unique one in Table 1 and Table 2. Moreover, in Table 2, $X_{k,1}(2,6)$, $X_{k,2}(2,6)$, $X_{k,3}(5,6)$, $X_{k,6}(p,6)$ are 3-arc-transitive and the others are 2-arc-transitive but not 3-arc-transitive.

Each connected arc-transitive finite abelian covering of $\Gamma$ is isomorphic to the fibered product of connected $p$-primary ones, with distinct primes $p$.
\end{thm}

\textbf{Table 1}

Elementary abelian coverings of the Petersen graphs

\footnotesize
\begin{center}
\begin{tabular}{|c|c|c|c|c|c|c|c|c|c|}
\hline
\multicolumn{1}{|c|}{Covering}&\multicolumn{1}{|c|}{$A$}&\multicolumn{1}{|c|}{$\phi^{t}(x_{1})$}&\multicolumn{1}{|c|}{$\phi^{t}(x_{2})$}
&\multicolumn{1}{|c|}{$\phi^{t}(x_{3})$}&\multicolumn{1}{|c|}{$\phi^{t}(x_{4})$}&\multicolumn{1}{|c|}{$\phi^{t}(x_{5})$}&\multicolumn{1}{|c|}{$\phi^{t}(x_{6})$}
&\multicolumn{1}{|c|}{Condition}&\multicolumn{1}{|c|}{$\text{Admissible}\atop\text{for}$} \\ \hline
$X(2,1)$  &$\mathbb{Z}_{2}$
&$\begin{array}{c}  1 \\  \end{array}$
&$\begin{array}{c}  1 \\  \end{array}$
&$\begin{array}{c}  1 \\  \end{array}$
&$\begin{array}{c}  1 \\  \end{array}$
&$\begin{array}{c}  1 \\  \end{array}$
&$\begin{array}{c}  1 \\  \end{array}$
& &$S_5$\\ \hline
$X'(2,1)$  &$\mathbb{Z}_{2}$
&$\begin{array}{c}  1 \\ \end{array}$
&$\begin{array}{c}  0 \\ \end{array}$
&$\begin{array}{c}  1 \\ \end{array}$
&$\begin{array}{c}  0 \\ \end{array}$
&$\begin{array}{c}  1 \\ \end{array}$
&$\begin{array}{c}  0 \\ \end{array}$
& &$A_5$\\ \hline
$X(2,2)$  &$\mathbb{Z}_{2}^{2}$
&$\begin{array}{c}  1 \\ 0 \\ \end{array}$
&$\begin{array}{c}  0 \\ 1 \\ \end{array}$
&$\begin{array}{c}  1 \\ 0 \\ \end{array}$
&$\begin{array}{c}  0 \\ 1 \\ \end{array}$
&$\begin{array}{c}  1 \\ 0 \\ \end{array}$
&$\begin{array}{c}  0 \\ 1 \\ \end{array}$
&  &$S_5$\\ \hline
$X(5,3)$  &$\mathbb{Z}_{5}^{3}$
&$\begin{array}{c}  1 \\ 0 \\ 0 \end{array}$
&$\begin{array}{c}  1 \\ 3 \\ 4 \end{array}$
&$\begin{array}{c}  0 \\ 0 \\ 1 \end{array}$
&$\begin{array}{c}  3 \\ 1 \\ 1 \end{array}$
&$\begin{array}{c}  0 \\ 1 \\ 0 \end{array}$
&$\begin{array}{c}  1 \\ 4 \\ 3 \end{array}$
& &$S_5$\\ \hline
$X^{\pm}(p,3)$  &$\mathbb{Z}_{p}^{3}$
&$\begin{array}{c}  1 \\ 0 \\ 0 \\ \end{array}$
&$\begin{array}{c}  1 \\ \lambda_{\pm} \\ -1 \\ \end{array}$
&$\begin{array}{c}  0 \\ 0 \\ 1 \\ \end{array}$
&$\begin{array}{c}  \lambda_{\pm} \\ 1 \\ 1 \\ \end{array}$
&$\begin{array}{c}  0 \\ 1 \\ 0 \\ \end{array}$
&$\begin{array}{c}  1 \\ -1 \\ \lambda_{\pm} \\ \end{array}$
&$p\equiv\pm 1\atop \pmod{10}$  &$A_5$\\  \hline
$X(p,6)$  &$\mathbb{Z}_{p}^{6}$
&$\begin{array}{c}  1 \\ 0 \\ 0 \\ 0 \\ 0 \\ 0 \\ \end{array}$
&$\begin{array}{c}  0 \\ 0 \\ 0 \\ 0 \\ 0 \\ 1 \\ \end{array}$
&$\begin{array}{c}  0 \\ 0 \\ 0 \\ 0 \\ 1 \\ 0 \\ \end{array}$
&$\begin{array}{c}  0 \\ 0 \\ 0 \\ 1 \\ 0 \\ 0 \\ \end{array}$
&$\begin{array}{c}  0 \\ 0 \\ 1 \\ 0 \\ 0 \\ 0 \\ \end{array}$
&$\begin{array}{c}  0 \\ 1 \\ 0 \\ 0 \\ 0 \\ 0 \\ \end{array}$
& $p \atop\text{arbitrary}$ &$S_5$\\ \hline
\end{tabular}
\end{center}

\vspace{4mm}

\newpage
\textbf{Table 2}

New families of abelian coverings of the Petersen graph
\footnotesize
\begin{center}
\begin{tabular}{|c|c|c|c|c|c|c|c|c|c|}
\hline
\multicolumn{1}{|c|}{Covering}&\multicolumn{1}{|c|}{$A$}&\multicolumn{1}{|c|}{$\phi^{t}_{1}$}&\multicolumn{1}{|c|}{$\phi^{t}_{2}$}
&\multicolumn{1}{|c|}{$\phi^{t}_{3}$}&\multicolumn{1}{|c|}{$\phi^{t}_{4}$}&\multicolumn{1}{|c|}{$\phi^{t}_{5}$}&\multicolumn{1}{|c|}{$\phi^{t}_{6}$}
&\multicolumn{1}{|c|}{Condition}&\multicolumn{1}{|c|} {$\text{admissible}\atop \text{for}$}\\ \hline
$X_{k,1}(2,6)$  & $\mathbb{Z}_{2^{k-1}}^{5}\times\mathbb{Z}_{2^{k}}$
&$\begin{array}{c}  -1 \\ 0 \\ 0 \\ 0 \\ 0 \\ 1 \\ \end{array}$
&$\begin{array}{c}  -1 \\ 0 \\ 0 \\ 0 \\ 1 \\ 0 \\ \end{array}$
&$\begin{array}{c}  -1 \\ 0 \\ 0 \\ 1 \\ 0 \\ 0 \\ \end{array}$
&$\begin{array}{c}  -1 \\ 0 \\ 1 \\ 1 \\ 0 \\ 0 \\ \end{array}$
&$\begin{array}{c}  -1 \\ 1 \\ 0 \\ 0 \\ 0 \\ 0 \\ \end{array}$
&$\begin{array}{c}  1 \\ 0 \\ 0 \\ 0 \\ 0 \\ 0 \\ \end{array}$
&$k>1$&$S_5$  \\ \hline
$X'_{k,1}(2,6)$  & $\mathbb{Z}_{2^{k-1}}^{5}\times\mathbb{Z}_{2^{k}}$
&$\begin{array}{c}  0 \\ 0 \\ 0 \\ 0 \\ 0 \\ 1 \\ \end{array}$
&$\begin{array}{c}  -1 \\ 0 \\ 0 \\ 0 \\ 1 \\ 0 \\ \end{array}$
&$\begin{array}{c}  0 \\ 0 \\ 0 \\ 1 \\ 0 \\ 0 \\ \end{array}$
&$\begin{array}{c}  -1 \\ 0 \\ 1 \\ 0 \\ 0 \\ 0 \\ \end{array}$
&$\begin{array}{c}  0 \\ 1 \\ 0 \\ 0 \\ 0 \\ 0 \\ \end{array}$
&$\begin{array}{c}  1 \\ 0 \\ 0 \\ 0 \\ 0 \\ 0 \\ \end{array}$
&$k>1$&$A_5$  \\ \hline
$X_{k,2}(2,6)$  & $\mathbb{Z}_{2^{k-1}}^{4}\times\mathbb{Z}_{2^{k}}^{2}$
&$\begin{array}{c}  0 \\ 1 \\ 0 \\ 0 \\ 0 \\ 1 \\ \end{array}$
&$\begin{array}{c}  1 \\ 0 \\ 0 \\ 0 \\ 1 \\ 0 \\ \end{array}$
&$\begin{array}{c}  0 \\ 1 \\ 0 \\ 1 \\ 0 \\ 0 \\ \end{array}$
&$\begin{array}{c}  1 \\ 0 \\ 1 \\ 0 \\ 0 \\ 0 \\ \end{array}$
&$\begin{array}{c}  0 \\ 1 \\ 0 \\ 0 \\ 0 \\ 0 \\ \end{array}$
&$\begin{array}{c}  1 \\ 0 \\ 0 \\ 0 \\ 0 \\ 0 \\ \end{array}$
&$k>1$&$S_5$  \\ \hline
$X^{\pm}_{k,3}(p,3)$  & $\mathbb{Z}_{p^{k}}^{3}$
&$\begin{array}{c}  0 \\ 0 \\ 1 \\ \end{array}$
&$\begin{array}{c}  0 \\ 1 \\ 0 \\ \end{array}$
&$\begin{array}{c}  1 \\ -\lambda_{\mp} \\ -\lambda_{\pm} \\ \end{array}$
&$\begin{array}{c}  \lambda_{\pm} \\ \lambda_{\pm} \\ -\lambda_{\pm} \\ \end{array}$
&$\begin{array}{c}  -\lambda_{\mp} \\ 1 \\ -\lambda_{\pm} \\ \end{array}$
&$\begin{array}{c}  1 \\ 0 \\ 0 \\ \end{array}$
&$k>1$&$A_5$  \\  \hline
$X_{k,3}(5,6)$  & $\mathbb{Z}_{5^{k-1}}^{3}\times\mathbb{Z}_{5^{k}}^{3}$
&$\begin{array}{c}  0 \\ 0 \\ 1 \\ 0 \\ 0 \\ 0 \\ \end{array}$
&$\begin{array}{c}  0 \\ 1 \\ 0 \\ 0 \\ 0 \\ 0 \\ \end{array}$
&$\begin{array}{c}  1 \\ 2 \\ 2 \\ 0 \\ 0 \\ 1 \\ \end{array}$
&$\begin{array}{c}  3 \\ 3 \\ 2 \\ 0 \\ 1 \\ 0 \\ \end{array}$
&$\begin{array}{c}  2 \\ 1 \\ 2 \\ 1 \\ 0 \\ 0 \\ \end{array}$
&$\begin{array}{c}  1 \\ 0 \\ 0 \\ 0 \\ 0 \\ 0 \\ \end{array}$
&$k>1$&$S_5$  \\ \hline
$X_{k,c,3}^{\pm}(p,6)$  &$\mathbb{Z}_{p^{c}}^{3}\times\mathbb{Z}_{p^{k}}^{3}$
&$\begin{array}{c}  0 \\ 0 \\ 1 \\ 0 \\ 0 \\ 0 \\ \end{array}$
&$\begin{array}{c}  0 \\ 1 \\ 0 \\ 0 \\ 0 \\ 0 \\ \end{array}$
&$\begin{array}{c}  1 \\ -\lambda_{\mp} \\ -\lambda_{\pm} \\ 0 \\ 0 \\ 1 \\ \end{array}$
&$\begin{array}{c}  \lambda_{\pm} \\ \lambda_{\pm} \\ -\lambda_{\pm} \\ 0 \\ 1 \\ 0 \\ \end{array}$
&$\begin{array}{c}  -\lambda_{\mp} \\ 1 \\ -\lambda_{\pm} \\ 1 \\ 0 \\ 0 \\ \end{array}$
&$\begin{array}{c}  1 \\ 0 \\ 0 \\ 0 \\ 0 \\ 0 \\ \end{array}$
&$k>c\geqslant 1$&$A_5$  \\ \hline
$X_{k,6}(p,6)$  &$\mathbb{Z}_{p^{k}}^{6}$
&$\begin{array}{c}  1 \\ 0 \\ 0 \\ 0 \\ 0 \\ 0 \\ \end{array}$
&$\begin{array}{c}  0 \\ 1 \\ 0 \\ 0 \\ 0 \\ 0 \\ \end{array}$
&$\begin{array}{c}  0 \\ 0 \\ 1 \\ 0 \\ 0 \\ 0 \\ \end{array}$
&$\begin{array}{c}  0 \\ 0 \\ 0 \\ 1 \\ 0 \\ 0 \\ \end{array}$
&$\begin{array}{c}  0 \\ 0 \\ 0 \\ 0 \\ 1 \\ 0 \\ \end{array}$
&$\begin{array}{c}  0 \\ 0 \\ 0 \\ 0 \\ 0 \\ 1 \\ \end{array}$
&$p\text{\ arbitrary}\atop k>1$& $S_5$  \\ \hline
\end{tabular}
\end{center}

\normalsize

\newpage

\subsection{Proof of Theorem \ref{thm:abelian coverings of Petersen}}

The proof is a good illustration of the method developed in Section 3.3, although some modifications are taken in order to simplify the process.

The trivial pair $(k,\mathbf{0})$ gives rise to $X_{k,6}(p,6)$ for any $p$ and any $k$.

By Theorem \ref{thm:recall},
\begin{align*}
\mathcal{P}_{0}=\{(2,5),(2,4),(5,3)\}\cup\{(p,3)\colon p\equiv\pm 1\pmod{10}\},
\end{align*}
hence
\begin{align}
\mathcal{P}&=\{(2;5),(2;4),(2;4,5),(2;0,4),(2;0,5),(2;0,4,5)\} \nonumber \\
&\cup\{(p;3),(p;0,3)\colon p=5\text{\ or\ }p\equiv\pm 1\pmod{10}\}.
\end{align}

Note that $\alpha_{1},\alpha_{3}, \alpha_{4}\in\textrm{Aut}_{G}(\Gamma;T)$; since $\alpha_{2}$ has order $2$, we can deduce that $\textrm{Aut}_{G}(\Gamma;T)=\langle\alpha_{1},\alpha_{3},\alpha_{4}\rangle$. In the notation of (\ref{eq:tau}),
\begin{align}
\tau^{\alpha_{1}}=(26)(35), \hspace{5mm} \tau^{\alpha_{3}}=(153)(264),\hspace{5mm} \tau^{\alpha_{4}}=(14)(25)(36).
\end{align}

Suppose $(p;k,l;r_{1},\cdots,r_{l};Q,\omega)$ is a $G$-admissible solution. We shall discuss various cases according to $i_{0}$ (the largest $i$ with $r_{i}=0$) and $l$.

Note that, since $\textrm{Aut}(\Gamma)/H\cong\{\bf{1},\alpha_{4}\}$, two solutions $(p;k,l;r_{1},\cdots,r_{l};Q,\omega)$ and $(p;k,l;r_{1},\cdots,r_{l};Q',\omega')$ are isomorphic if and only if $\langle PQ\omega\rangle=\langle PQ'\omega'\rangle$ or $(\alpha_{4})_{\ast}\langle PQ\omega\rangle=\langle PQ'\omega\rangle$, where $P\in\mathbb{Z}_{p^k}^{l,6},P_{i,j}=\delta_{i,j}\cdot p^{r_{i}}$.

Remark \ref{rmk:from solution to covering} tells us how to explicitly construct the corresponding covering from a solution. And for each covering, it is easy to check whether it is $S_{5}$-admissible by testing whether it is possible to lift $\alpha_{4}$, in a routine way.

\subsubsection{$i_{0}=5$}

Then $p=2$, $l=5$, $0=r_{1}=\cdots =r_{5}<r_{6}=k$.

First, we show that $\Sigma_{6}(5)=\{[\textrm{id}]\}$. Each element of $\Sigma_{6}$ is equivalent to some $\omega$ with $\omega^{-1}(1)<\cdots <\omega^{-1}(5)$; there are 6 such elements
$$\textrm{id}, \omega_{1}=(654321), \omega_{2}=(65432), \omega_{3}=(6543), \omega_{4}=(654), \omega_{5}=(65).$$
It is easy to verify that $\textrm{id}\overset{\alpha_{4}}{\sim}\omega_{3}\overset{\alpha_{3}}{\sim}\omega_{1}\overset{\alpha_{4}}{\sim}\omega_{4}\overset{\alpha_{3}}{\sim}\omega_{2}\overset{\alpha_{4}}{\sim}\omega_{5}$
, where by $\omega\overset{\alpha}{\sim}\omega'$ we mean (\ref{eq:equivalence}) holds for $\beta=\alpha$.

Thus $\Sigma_{6}(5)=\{[\textrm{id}]\}$ and we can assume $\omega=\textrm{id}$.

In the notation of Remark \ref{rmk:congruence equation}, suppose $\Theta=(q_{1},\cdots,q_{5})^{t}$. The equation (\ref{eq:congruence equation}) for $\alpha=\alpha_{1},\alpha_{2},\alpha_{3}$ can be written as, respectively,
\begin{align}
\left( \begin{array}{ccccc}
      -1 & -q_{1} & 0 & 0 & 0 \\
      0 & -q_{2} & 0 & 0 & 0 \\
      0 & -q_{3} & 0 & 0 & -1 \\
      0 & -q_{4} & 0 & -1 & 0 \\
      0 & -q_{5} & -1 & 0 & 0 \\
       \end{array}
\right) \left(
       \begin{array}{c} q_{1} \\ q_{2} \\ q_{3} \\  q_{4} \\ q_{5} \\
       \end{array}
\right) &=\left(
       \begin{array}{c} 0 \\ -1 \\ 0 \\ 0 \\ 0 \\ \end{array}\right), \label{eq:(1.1)-1} \\
\left( \begin{array}{ccccc}
      q_{1} & 0 & 0 & 0 & -1-q_{1} \\
      q_{2} & -1 & 0 & 0 & -q_{2} \\
      q_{3} & 1 & 1 & -1 & -q_{3} \\
      q_{4} & 0 & 0 & -1 & -q_{4} \\
      -1+q_{5} & 0 & 0 & 0 & -q_{5} \\
       \end{array}
\right) \left(
       \begin{array}{c} q_{1} \\ q_{2} \\ q_{3} \\ q_{4} \\ q_{5} \\
       \end{array}
\right) &= -\left(
       \begin{array}{c} q_{1} \\ q_{2} \\ q_{3} \\ q_{4} \\ q_{5} \\ \end{array}\right), \label{eq:(1.1)-2} \\
\left( \begin{array}{ccccc}
      0 & q_{1} & -1 & 0 & 0 \\
      0 & q_{2} & 0 & -1 & 0 \\
      0 & q_{3} & 0 & 0 & 1 \\
      0 & q_{4} & 0 & 0 & 0 \\
      -1 & q_{5} & 0 & 0 & 0 \\
       \end{array}
\right) \left(
       \begin{array}{c} q_{1} \\ q_{2} \\ q_{3} \\ q_{4} \\ q_{5} \\
       \end{array}
\right) &=\left(
       \begin{array}{c} 0 \\ 0 \\ 0 \\ -1 \\ 0\\ \end{array}\right), \label{eq:(1.1)-3}
\end{align}
which hold in $\mathbb{Z}_{2^{k}}$.

From (\ref{eq:(1.1)-1})-(\ref{eq:(1.1)-3}) one can deduce that $k=1$, $q_{2}=q_{4}=1$, $q_{1}=q_{3}=q_{5}$. Hence there are two possibilities: $\Theta=(1,1,1,1,1)^{t}$ or $\Theta=(0,1,0,1,0)^{t}$.
Let
\begin{align}
Q_{1}=\left(
        \begin{array}{cccccc}
          1 & 0 & 0 & 0 & 0 & 1 \\
          0 & 1 & 0 & 0 & 0 & 1 \\
          0 & 0 & 1 & 0 & 0 & 1 \\
          0 & 0 & 0 & 1 & 0 & 1 \\
          0 & 0 & 0 & 0 & 1 & 1 \\
          0 & 0 & 0 & 0 & 0 & 1 \\
        \end{array}
      \right),\hspace{5mm}
Q'_{1}=\left(
        \begin{array}{cccccc}
          1 & 0 & 0 & 0 & 0 & 0 \\
          0 & 1 & 0 & 0 & 0 & 1 \\
          0 & 0 & 1 & 0 & 0 & 0 \\
          0 & 0 & 0 & 1 & 0 & 1 \\
          0 & 0 & 0 & 0 & 1 & 0 \\
          0 & 0 & 0 & 0 & 0 & 1 \\
        \end{array}
      \right).
\end{align}
It is clear that the solution $(2;1,5;0,0,0,0,0;Q_{1},\textrm{id})$ gives rise to $X(2,1)$, and $(2;1,5;0,0,0,0,0;Q'_{1},\textrm{id})$ gives rise to a covering isomorphic to $X'(2,1)$.

\subsubsection{$i_{0}=4$}

Then $p=2$, $l=4$ or $5$, and $0=r_{1}=\cdots =r_{4}<r_{5}\leqslant r_{6}=k$.

(a) When $l=4$, $r_{5}=r_{6}$.

$\Sigma_{6}(4)=\{[\textrm{id}],[(45)],[(354)]\}$. Let $\Theta=(q_{ij})_{4\times 2}$.

(i) When $\omega=\textrm{id}$.

We have the following equations in $\mathbb{Z}_{2^{k}}$, obtained from (\ref{eq:congruence equation}):
\begin{align}
\left(
   \begin{array}{cccc}
     -1 & -q_{12} & -q_{11} & 0 \\
     0 & -q_{22} & -q_{21} & 0 \\
     0 & -q_{32} & -q_{31} & 0 \\
     0 & -q_{42} & -q_{41} & -1 \\
   \end{array}
 \right)\left(
          \begin{array}{cc}
            q_{11} & q_{12} \\
            q_{21} & q_{22} \\
            q_{31} & q_{32} \\
            q_{41} & q_{42} \\
          \end{array}
        \right)&=\left(
                  \begin{array}{cc}
                    0 & 0 \\
                    0 & -1 \\
                    -1 & 0 \\
                    0 & 0 \\
                  \end{array}
                \right)
, \label{eq:(2.1)-1}
\\
\left(
   \begin{array}{cccc}
     q_{12}-q_{11} & 0 & 0 & 0 \\
     q_{22}-q_{21} & -1 & 0 & 0 \\
     q_{32}-q_{31} & 1 & 1 & -1 \\
     q_{42}-q_{41} & 0 & 0 & -1 \\
   \end{array}
 \right)\left(
          \begin{array}{cc}
            q_{11} & q_{12} \\
            q_{21} & q_{22} \\
            q_{31} & q_{32} \\
            q_{41} & q_{42} \\
          \end{array}
        \right)&=\left(
                  \begin{array}{cc}
                    -1-q_{12} & -q_{12} \\
                    -q_{22} & -q_{22} \\
                    -q_{32} & -q_{32} \\
                    -q_{42} & -q_{42} \\
                  \end{array}
                \right)
,\label{eq:(2.1)-2}
\\
\left(
   \begin{array}{cccc}
     -q_{11} & q_{12} & -1 & 0 \\
     -q_{21} & q_{22} & 0 & -1 \\
     -q_{31} & q_{32} & 0 & 0 \\
     -q_{41} & q_{42} & 0 & 0 \\
   \end{array}
 \right)\left(
          \begin{array}{cc}
            q_{11} & q_{12} \\
            q_{21} & q_{22} \\
            q_{31} & q_{32} \\
            q_{41} & q_{42} \\
          \end{array}
        \right)&=\left(
                  \begin{array}{cc}
                    0 & 0 \\
                    0 & 0 \\
                    1 & 0 \\
                    0 & -1 \\
                  \end{array}
                \right)
.\label{eq:(2.1)-3}
\end{align}

Use (\ref{eq:(2.1)-2})-(1,1) to denote the equation obtained by comparing the (1,1)-entries of (\ref{eq:(2.1)-2}) and so on.

We have the following deductions: \\
(\ref{eq:(2.1)-2})-(1,1),(\ref{eq:(2.1)-2})-(2,2),(\ref{eq:(2.1)-3})-(1,1), (\ref{eq:(2.1)-3})-(1,2)$\Rightarrow$ $q_{32}-q_{31}=1+q_{12}$; \\
(\ref{eq:(2.1)-1})-(1,1), (\ref{eq:(2.1)-1})-(1,2) $\Rightarrow$ $q_{12}(1+q_{11})=0$; \\
(\ref{eq:(2.1)-2})-(1,1), (\ref{eq:(2.1)-2})-(1,2)$\Rightarrow$ $q_{11}^{2}=1, q_{12}^{2}=-2q_{12},(q_{12}-q_{11})^2=1$;  \\
(\ref{eq:(2.1)-2})-(2,2)$\Rightarrow$ $q_{12}(q_{22}-q_{21})=0$;\\
(\ref{eq:(2.1)-3})-(2,1), (\ref{eq:(2.1)-3})-(2,2)$\Rightarrow$ $q_{42}-q_{41}=q_{21}q_{11}-q_{22}q_{21}-q_{21}q_{12}+q_{22}^2$; \\
(\ref{eq:(2.1)-2})-(4,2)$\Rightarrow$ $0=q_{12}(q_{42}-q_{41})=q_{12}(q_{21}q_{11}-q_{22}q_{21}-q_{21}q_{12}+q_{22}^2)=q_{12}q_{21}(q_{11}-q_{12})+q_{22}q_{12}(q_{22}-q_{21})=q_{12}q_{21}(q_{11}-q_{12})\Rightarrow q_{12}q_{21}=0,q_{12}q_{22}=0$;\\
(\ref{eq:(2.1)-3})-(1,1)$\Rightarrow$ $q_{31}=-1$; \\
(\ref{eq:(2.1)-3})-(1,2)$\Rightarrow$ $q_{32}=q_{12}$; \\
(\ref{eq:(2.1)-1})-(2,2)$\Rightarrow$ $q_{22}^{2}=1$, $q_{12}=0$; \\
(\ref{eq:(2.1)-3})-(2,2)$\Rightarrow$ $q_{42}=1$; \\
(\ref{eq:(2.1)-3})-(3,2)$\Rightarrow$ $q_{32}=0$; \\
(\ref{eq:(2.1)-3})-(4,2)$\Rightarrow$ $q_{22}=-1$; \\
(\ref{eq:(2.1)-1})-(4,1)$\Rightarrow$ $q_{21}=0$; \\
(\ref{eq:(2.1)-2})-(2,1)$\Rightarrow$ $q_{11}=-1$; \\
(\ref{eq:(2.1)-2})-(3,1)$\Rightarrow$ $q_{41}=0$; \\
(\ref{eq:(2.1)-2})-(3,2)$\Rightarrow$ $2=0$ $\Rightarrow k=1$.

Thus the solution is $(2;1,4;0,0,0,0;Q_{2},\textrm{id})$ with
\begin{align}
Q_{2}=\left(
           \begin{array}{cccccc}
               1 & 0 & 0 & 0 & 1 & 0 \\
               0 & 1 & 0 & 0 & 0 & 1 \\
               0 & 0 & 1 & 0 & 1 & 0 \\
               0 & 0 & 0 & 1 & 0 & 1 \\
               0 & 0 & 0 & 0 & 1 & 0 \\
               0 & 0 & 0 & 0 & 0 & 1 \\
           \end{array}
       \right).
\end{align}
It can be checked that the $S_{5}$-admissible covering determined by this solution is isomorphic to $X(2,2)$.

(ii) When $\omega=(45)$.

The equation (\ref{eq:congruence equation}) over $\mathbb{Z}_{2^{k}}$ for $\alpha=\alpha_{1}, \alpha_{2}$ reads respectively
\begin{align}
\left(
  \begin{array}{cccc}
    -1 & -q_{12} & 0 & 0 \\
    0 & -q_{22} & 0 & 0 \\
    0 & -q_{32} & 0 & -1 \\
    0 & -q_{42} & -1 & 0 \\
  \end{array}
  \right)\left(
         \begin{array}{cc}
           q_{11} & q_{12} \\
           q_{21} & q_{22} \\
           q_{31} & q_{32} \\
           q_{41} & q_{42} \\
         \end{array}
       \right)
  &=\left(
   \begin{array}{cc}
     -q_{11} & 0 \\
     -q_{21} & -1 \\
     -q_{31} & 0 \\
     -q_{41} & 0 \\
   \end{array}
   \right), \label{eq:(2.1)-ii-1} \\
\left(
  \begin{array}{cccc}
    q_{12} & 0 & 0 & -1-q_{12} \\
    q_{22} & -1 & 0 & -q_{22} \\
    q_{32} & 1 & 1 & -q_{32} \\
    q_{42}-1 & 0 & 0 & -q_{42} \\
  \end{array}
  \right)\left(
         \begin{array}{cc}
           q_{11} & q_{12} \\
           q_{21} & q_{22} \\
           q_{31} & q_{32} \\
           q_{41} & q_{42} \\
         \end{array}
       \right)
  &=\left(
   \begin{array}{cc}
     -q_{11} & -q_{12} \\
     -q_{21} & -q_{22} \\
     -1-q_{31} & -q_{32} \\
     -q_{41} & -q_{42} \\
   \end{array}
 \right). \label{eq:(2.1)-ii-2}
\end{align}
(\ref{eq:(2.1)-ii-1})-(2,2)$\Rightarrow$ $q_{22}^{2}=1$; \\
(\ref{eq:(2.1)-ii-2})-(2,2)$\Rightarrow$ $q_{22}(q_{12}-q_{42})=0\Rightarrow q_{12}=q_{42}$; \\
(\ref{eq:(2.1)-ii-2})-(3,2)$\Rightarrow$ $q_{22}+2q_{32}=0$. \\
This leads to contradiction. Hence in this case there is no solution.

(iii) When $\omega=(354)$.

Similarly, the equations (\ref{eq:congruence equation}) for $\alpha=\alpha_{1}, \alpha_{3}$ lead to contradiction.

\vspace{4mm}

(b) When $l=5$, $r_{5}<r_{6}=k$.

It turns out that $\Sigma_{6}(4,5)=\{[\textrm{id}], [(45)], [(354)]\}=\Sigma_{6}(4)$, so we still only need to consider $\omega=\textrm{id}, (45), (354)$.

Suppose $Q_{i,j}=q_{ij}, (i=1,\cdots,4,j=5,6)$, $Q_{5,6}=q$. From Proposition \ref{prop:key2} and the result of (2.1), it follows that $\omega=\textrm{id}$, $r_{5}=1$ and $Q\pmod{2}=Q_{2}$; by Proposition \ref{prop:key2} and the result of Section 4.3.1, we have $r_{6}=1$ and
$$\left(
   \begin{array}{ccccc}
     1 & 0 & 0 & 0 & -q_{11} \\
     0 & 1 & 0 & 0 & -q_{21} \\
     0 & 0 & 1 & 0 & -q_{31} \\
     0 & 0 & 0 & 1 & -q_{41} \\
     0 & 0 & 0 & 0 & 1 \\
   \end{array}
 \right)\left(
          \begin{array}{c}
            q_{12} \\
            q_{22} \\
            q_{32} \\
            q_{42} \\
            q \\
          \end{array}
        \right)
=\left(
   \begin{array}{c}
     1 \\
     0 \\
     1 \\
     0 \\
     1 \\
   \end{array}
 \right) \text{\ or\ }
\left(
   \begin{array}{c}
     1 \\
     1 \\
     1 \\
     1 \\
     1 \\
   \end{array}
 \right).$$
Hence $q_{11}=q_{31}=q=1$, $q_{21}=q_{41}=0$, $q_{12}, q_{32}\in\{0,2\}, q_{22}, q_{42}\in\{1,3\}$.
The equations $(QS^{\alpha}Q^{-1})_{i,6}\equiv 0\pmod{4}$ for $1\leqslant i\leqslant 4$ and $\alpha=\alpha_{1},\alpha_{2},\alpha_{3}$ imply
\begin{align*}
&\left(
    \begin{array}{c}
      q_{12}+q_{32}+q_{12}q_{22}-2 \\
      q_{22}^{2}-1 \\
      q_{32}(1+q_{22}) \\
      q_{42}(1+q_{22}) \\
    \end{array}
  \right)
\equiv\left(
    \begin{array}{c}
      q_{12}+q_{32}+q_{12}q_{22}-2 \\
      q_{42}-q_{22}^{2} \\
      q_{12}-q_{22}q_{32}-2 \\
      -1-q_{22}q_{42} \\
    \end{array}
  \right) \\
\equiv &\left(
    \begin{array}{c}
      1-(1-q_{12})^{2} \\
      q_{22}(2-q_{12}) \\
      q_{12}-q_{22}q_{32}-2 \\
      q_{12}-q_{22}+q_{42}-q_{12}q_{32} \\
    \end{array}
  \right)
\equiv\left(
    \begin{array}{c}
      0 \\
      0 \\
      0 \\
      0 \\
    \end{array}
  \right),
\end{align*}
which is impossible to hold, as can be checked easily. Thus there is no solution.

\subsubsection{$i_{0}=3$}

Then $p=5$, or $p\equiv\pm 1\pmod{10}$,
$l=3$, $0=r_{1}=r_{2}=r_{3}<r_{4}=r_{5}=r_{6}=k$, $\Sigma_{6}(3)=\{[(14253)], [(34)], [(2453)]\}$. Let $\Theta=(q_{ij})_{3\times 3}$.

(i) $\omega=(14253)$.

We have the following equations over $\mathbb{Z}_{p^{k}}$:
\begin{align}
\left(
  \begin{array}{ccc}
    0 & 0 & 1 \\
    0 & 1 & 0 \\
    1 & 0 & 0 \\
  \end{array}
\right)\left(
         \begin{array}{ccc}
           q_{11} & q_{12} & q_{13} \\
           q_{21} & q_{22} & q_{23} \\
           q_{31} & q_{32} & q_{33} \\
         \end{array}
       \right)
&=\left(
  \begin{array}{ccc}
    q_{11} & q_{13} & q_{12} \\
    q_{21} & q_{23} & q_{22} \\
    q_{31} & q_{33} & q_{32} \\
  \end{array}
\right),\label{eq:(3)-i-1}\\
\left(
  \begin{array}{ccc}
    1 & -1 & -q_{11}-q_{13} \\
    0 & -1 & -q_{21}-q_{23} \\
    0 & 0  & -q_{31}-q_{33} \\
  \end{array}
\right)\left(
         \begin{array}{ccc}
           q_{11} & q_{12} & q_{13} \\
           q_{21} & q_{22} & q_{23} \\
           q_{31} & q_{32} & q_{33} \\
         \end{array}
       \right)
&=\left(
  \begin{array}{ccc}
    q_{13} & 1-q_{12} & -q_{13} \\
    q_{23} & -q_{22} & -q_{23} \\
    q_{33}-1 & -q_{32} & -q_{33} \\
  \end{array}
\right),\label{eq:(3)-i-2} \\
\left(
  \begin{array}{ccc}
    -q_{11} & -q_{12} & 1 \\
    -q_{21} & -q_{22} & 0 \\
    -q_{31} & -q_{32} & 0 \\
  \end{array}
\right)\left(
         \begin{array}{ccc}
           q_{11} & q_{12} & q_{13} \\
           q_{21} & q_{22} & q_{23} \\
           q_{31} & q_{32} & q_{33} \\
         \end{array}
       \right)
&=\left(
  \begin{array}{ccc}
    0 & q_{13} & 0 \\
    0 & q_{23} & -1 \\
   -1 & q_{33} & 0 \\
  \end{array}
\right) \label{eq:(3)-i-3}.
\end{align}
(\ref{eq:(3)-i-1}) $\Rightarrow$ $q_{31}=q_{11}$, $q_{32}=q_{13}$, $q_{33}=q_{12}$, $q_{23}=q_{22}$; \\
(\ref{eq:(3)-i-3})-(3,2),(3,3)$\Rightarrow$ $q_{11}(q_{13}-q_{12})=q_{12}, q_{13}(q_{11}+q_{22})=0$; \\
(\ref{eq:(3)-i-3})-(1,2)$\Rightarrow$ $q_{12}(q_{11}+q_{22})=0$; \\
(\ref{eq:(3)-i-3})-(2,3)$\Rightarrow$ $q_{21}q_{13}+q_{22}^{2}=1$; \\
(\ref{eq:(3)-i-2})-(1,3)$\Rightarrow$ $2q_{13}=q_{22}+q_{12}(q_{11}+q_{13})\Rightarrow$ $q_{12}$ or $q_{13}$ is invertible $\Rightarrow q_{22}=-q_{11}$; \\
(\ref{eq:(3)-i-2})-(2,2), (2,3)$\Rightarrow$ $(q_{21}+q_{23})q_{13}=0=(q_{21}+q_{23})q_{12}\Rightarrow q_{21}=-q_{23}=q_{11}$; \\
(\ref{eq:(3)-i-3})-(1,1)$\Rightarrow$ $q_{12}=1-q_{11}$; \\
(\ref{eq:(3)-i-3})-(2,3), (\ref{eq:(3)-i-2})-(1,3)$\Rightarrow$ $q_{13}=-1$. \\
Thus there is a solution if and only $q_{11}^{2}-q_{11}-1=0$, which is equivalent to
\begin{align}
x^{2}=5, \label{eq:quadratic}
\end{align}
with $x=2q_{11}-1$.

If $p=5$, it is easy to see that (\ref{eq:quadratic}) has a solution if and only if $k=1$.

If $p\neq 5$,
by Proposition 5.1.1 of \cite{GTM84} (page 50), the equation (\ref{eq:quadratic}) has a solution in $\mathbb{Z}_{p^{k}}$ if and only if the Legendre symbol $(5/p)$ is equal to 1, which turns out to be equivalent to $p\equiv \pm 1\pmod{10}$, and then there are exactly two solutions. Let  $\lambda_{\pm}\in\mathbb{Z}_{p^{k}}$ be the two solutions of $\lambda^{2}-\lambda-1=0$.

The solutions are $(5;1,3;0,0,0;Q_{3},(14253))$ and $(p;k,3;0,0,0;Q_{3}^{\pm},(14253))$ ($p\equiv\pm 1\pmod{10}, k\geqslant 1$), with
\begin{align}
Q_{3}=\left(
        \begin{array}{cccccc}
          1 & 0 & 0 & 3 & 3 & 4 \\
          0 & 1 & 0 & 3 & 2 & 2 \\
          0 & 0 & 1 & 3 & 4 & 3 \\
          0 & 0 & 0 & 1 & 0 & 0 \\
          0 & 0 & 0 & 0 & 1 & 0 \\
          0 & 0 & 0 & 0 & 0 & 1 \\
        \end{array}
      \right)\in\mathbb{Z}_{5}^{6,6},
\end{align}
\begin{align}
Q_{3}^{\pm}=\left(
        \begin{array}{cccccc}
          1 & 0 & 0 & \lambda_{\pm} & \lambda_{\mp} & -1 \\
          0 & 1 & 0 & \lambda_{\pm} & -\lambda_{\pm} & -\lambda_{\pm} \\
          0 & 0 & 1 & \lambda_{\pm} & -1 & \lambda_{\mp} \\
          0 & 0 & 0 & 1 & 0 & 0 \\
          0 & 0 & 0 & 0 & 1 & 0 \\
          0 & 0 & 0 & 0 & 0 & 1 \\
        \end{array}
      \right)\in\mathbb{Z}_{p^{k}}^{6,6}.
\end{align}

The solution $(5;1,3;0,0,0;Q_{3},(14253))$ gives rise to a covering isomorphic to $X(5,3)$, and the solution $(p;k,3;0,0,0;Q_{3}^{\pm},(14253))$ gives rise to $X^{\pm}_{k,3}(p,3)$ which can be verified to be not 3-arc-transitive; also, $X^{\pm}_{1,3}(p,3)$ is isomorphic to $X^{\pm}(p,3)$.

(ii) $\omega=(34)$.
\begin{align}
\left(
  \begin{array}{ccc}
    1 & q_{13} & 0 \\
    0 & q_{23} & 0 \\
    0 & q_{33} & -1 \\
  \end{array}
\right)\left(
         \begin{array}{ccc}
           q_{11} & q_{12} & q_{13} \\
           q_{21} & q_{22} & q_{23} \\
           q_{31} & q_{32} & q_{33} \\
         \end{array}
       \right)
&=\left(
  \begin{array}{ccc}
    q_{12} & q_{11} & 0 \\
    q_{22} & q_{21} & 1 \\
    q_{32} & q_{31} & 0 \\
  \end{array}
\right), \label{eq:(3)-ii-1} \\
\left(
  \begin{array}{ccc}
    q_{12}-q_{13} & -q_{11} & q_{11} \\
    q_{22}-q_{23} & 1-q_{21} & -q_{21} \\
    q_{32}-q_{33} & -q_{31} & q_{31}+1 \\
  \end{array}
\right)\left(
         \begin{array}{ccc}
           q_{11} & q_{12} & q_{13} \\
           q_{21} & q_{22} & q_{23} \\
           q_{31} & q_{32} & q_{33} \\
         \end{array}
       \right)
&=\left(
  \begin{array}{ccc}
    -q_{11} & q_{13}+1 & q_{13} \\
    -q_{21} & q_{23} & q_{23} \\
    -q_{31} & q_{33} & q_{33} \\
  \end{array}
\right), \label{eq:(3)-ii-2} \\
\left(
  \begin{array}{ccc}
    -q_{12} & q_{13} & 0 \\
    -q_{22} & q_{23} & -1 \\
    -q_{32} & q_{33} & 0 \\
  \end{array}
\right)\left(
         \begin{array}{ccc}
           q_{11} & q_{12} & q_{13} \\
           q_{21} & q_{22} & q_{23} \\
           q_{31} & q_{32} & q_{33} \\
         \end{array}
       \right)
&=\left(
  \begin{array}{ccc}
   -1 & q_{11} & 0 \\
    0 & q_{21} & 0 \\
    0 & q_{31} & -1 \\
  \end{array}
\right). \label{eq:(3)-ii-3}
\end{align}
(\ref{eq:(3)-ii-1})-(2,3)$\Rightarrow$ $q_{23}=\pm 1$.

If $q_{23}=1$. \\
(\ref{eq:(3)-ii-1})-(1,3)$\Rightarrow$ $q_{13}=0$; \\
(\ref{eq:(3)-ii-1})-(1,2)$\Rightarrow$ $q_{22}=q_{11}$; \\
(\ref{eq:(3)-ii-3})-(1,1),(1,2)$\Rightarrow$ $q_{11}=-1$; \\
(\ref{eq:(3)-ii-2})-(1,3)$\Rightarrow$ $q_{33}=q_{23}=1$;\\
(\ref{eq:(3)-ii-3})-(3,3)$\Rightarrow 1=-1$, a contradiction;

if $q_{23}=-1$. \\
(\ref{eq:(3)-ii-1})-(2,1)$\Rightarrow$ $q_{22}=-q_{21}$; \\
(\ref{eq:(3)-ii-1})-(3,3)$\Rightarrow$ $q_{33}=0$; \\
(\ref{eq:(3)-ii-1})-(3,2)$\Rightarrow$ $q_{31}=-q_{32}$; \\
(\ref{eq:(3)-ii-3})-(3,3)$\Rightarrow$ $q_{32}q_{13}=1$; \\
(\ref{eq:(3)-ii-2})-(3,3)$\Rightarrow$ $q_{32}=1$;\\
(\ref{eq:(3)-ii-2})-(3,1)$\Rightarrow$ $q_{22}=q_{11}-1$; \\
(\ref{eq:(3)-ii-2})-(3,2)$\Rightarrow$ $q_{12}=-q_{22}$; \\
(\ref{eq:(3)-ii-2})-(1,3)$\Rightarrow$ $0=1$, a contradiction.

Thus there is no solution.

(iii) $\omega=(2453)$.
\begin{align}
\left(
  \begin{array}{ccc}
    1 & 0 & 0 \\
    0 & 0 & 1 \\
    0 & 1 & 0 \\
  \end{array}
\right)\left(
         \begin{array}{ccc}
           q_{11} & q_{12} & q_{13} \\
           q_{21} & q_{22} & q_{23} \\
           q_{31} & q_{32} & q_{33} \\
         \end{array}
       \right)
&=\left(
  \begin{array}{ccc}
    q_{13} & q_{12} & q_{11} \\
    q_{23} & q_{22} & q_{21} \\
    q_{33} & q_{32} & q_{31} \\
  \end{array}
\right),\label{eq:(3)-iii-1}\\
\left(
  \begin{array}{ccc}
    -q_{13} & 0 & 1+q_{13} \\
    -q_{23} & -1 & q_{13} \\
    1-q_{33} & 0 & q_{33} \\
  \end{array}
\right)\left(
         \begin{array}{ccc}
           q_{11} & q_{12} & q_{13} \\
           q_{21} & q_{22} & q_{23} \\
           q_{31} & q_{32} & q_{33} \\
         \end{array}
       \right)
&=\left(
  \begin{array}{ccc}
    q_{11} & q_{12} & q_{13} \\
    q_{21}-1 & 1+q_{22} & q_{23} \\
    q_{31} & q_{32} & q_{33} \\
  \end{array}
\right),\label{eq:(3)-iii-2}\\
\left(
  \begin{array}{ccc}
    0 & -1 & 0 \\
    0 & 0 & 1 \\
    -1 & 0 & 0 \\
  \end{array}
\right)\left(
         \begin{array}{ccc}
           q_{11} & q_{12} & q_{13} \\
           q_{21} & q_{22} & q_{23} \\
           q_{31} & q_{32} & q_{33} \\
         \end{array}
       \right)
&=\left(
  \begin{array}{ccc}
    q_{13} & -q_{11} & -q_{12} \\
    q_{23} & -q_{21} & -q_{22} \\
    q_{33} & -q_{31} & -q_{32} \\
  \end{array}
\right). \label{eq:(3)-iii-3}
\end{align}
(\ref{eq:(3)-iii-2})-(1,3)$\Rightarrow$ $(q_{13}+1)(q_{13}-q_{33})=0$; \\
(\ref{eq:(3)-iii-2})-(2,1)$\Rightarrow$ $2q_{21}-1=q_{13}q_{31}-q_{11}q_{23}$; \\
(\ref{eq:(3)-iii-1})$\Rightarrow$ $q_{13}=q_{11}$, $q_{23}=q_{31}$, $q_{21}=q_{33}$, hence $q_{21}=1/2$, $(q_{11}+1)(q_{11}-1/2)=0$. But (\ref{eq:(3)-iii-3})-(1,1)$\Rightarrow$ $q_{13}=-q_{21}=-1/2$. So there is no solution.

\subsubsection{$i_{0}=0$}

By Remark \ref{rmk:reduction} and the previous results, the solutions are as follows: \\
$(2;k,5;k-1,k-1,k-1,k-1,k-1;Q_{1},\textrm{id})$, giving rise to $X_{k,1}(2,6)$; \\
$(2;k,5;k-1,k-1,k-1,k-1,k-1;Q'_{1},\textrm{id}), k>1$, giving rise to $X'_{k,1}(2,6)$; \\
$(2;k,4;k-1,k-1,k-1,k-1;Q_{2},\textrm{id}), k>1$, giving rise to $X_{k,2}(2,6)$; \\
$(5;k,3;k-1,k-1,k-1;Q_{3},(14253)), k>1$, giving rise to $X_{k,3}(5,6)$; \\
$(p;k,3;c,c,c;Q^{\pm}_{3},(14253)), k>1, p\equiv\pm 1\pmod{10}$, giving rise to $X^{\pm}_{k,c,3}(p,6)$. \\

\end{document}